\documentclass[reqno,twoside,12pt]{amsart}
\usepackage{amsthm,amsfonts}
\usepackage{amsmath,amssymb,amsfonts}
\usepackage{xcolor}

\usepackage[pagebackref=true, colorlinks=true, citecolor=blue, pdfborder={0 0 .1},breaklinks]{hyperref}

\usepackage[left=3.5cm,right=3.5cm, top=3.5cm,bottom=3.5cm]{geometry}

\newtheorem{definition}{\bf Definition}[section]
\newtheorem{lemma}{\bf Lemma}[section]
\newtheorem{proposition}{\bf Propostion}[section]
\newtheorem{theorem}{\bf Theorem}[section]

\newtheorem{step}{\bf Step}
\newtheorem{remark}{\bf Remark}[section]

\numberwithin{equation}{section}

\newcommand{\torus}{{\mathbb T}}

\newcommand{\ds}{\displaystyle}
\newcommand{\loc}{\rm{loc}}

\DeclareMathOperator{\dv}{div} %
\DeclareMathOperator{\curl}{curl} %

\title[Energy balance in 2D]{Energy balance for forced two-dimensional incompressible ideal fluid flow}

\author[Lopes Filho, Nussenzveig Lopes]{Milton C. Lopes Filho and Helena J. Nussenzveig Lopes}
\address{Instituto de Matem\'atica\\Universidade Federal do Rio de Janeiro\\Cidade Universit\'aria -- Ilha do Fund\~ao\\Caixa Postal 68530\\21941-909 Rio de Janeiro, RJ -- BRAZIL\\}
\email{mlopes@im.ufrj.br}
\email{hlopes@im.ufrj.br}

\begin{document}




\begin{abstract}
In \cite{CLNS16}, Cheskidov {\it et al.} proved that physically realizable weak solutions of the incompressible 2D Euler equations on a torus conserve kinetic energy. Physically realizable weak solutions are those that can be obtained as limits of vanishing viscosity. The key hypothesis was boundedness of the initial vorticity in $L^p$, $p>1$. In this work we extend their result, by adding forcing to the flow.
\end{abstract}

\maketitle

\begin{center}
{\em  To Uriel Frisch, on the occasion of his 80$^{th}$ birthday.}
\end{center}

\vspace{10pt}

\section{Introduction}

In this article we consider weak solutions of the incompressible Euler equations on the two-dimensional flat torus which happen to be vanishing viscosity limits in a sense that will be made precise. In \cite{CLNS16}, Cheskidov {\it et al.} proved that, under the assumption that the initial vorticities are $p$-th power integrable, with $p>1$, such solutions conserve kinetic energy. This result was extended in \cite{LMPP20} and  \cite{Ciampa21} and shown to be sharp in \cite{CL21}.  The flows considered in \cite{CLNS16}, \cite{CL21}, \cite{Ciampa21} and \cite{LMPP20} are unforced, so that conservation of kinetic energy means that the $L^2$-norm of velocity is conserved in time. For flows with forcing, conservation of energy takes the form of an identity, where the rate-of-change of the kinetic energy of the fluid equals the work performed by the external force on the fluid per unit of time.

In this work we seek precise regularity conditions on the forcing which allow us to extend the result on conservation of kinetic energy in \cite{CLNS16} to full energy conservation in flows with forcing.

There are three main motivations for this work. The first is that low-regularity flows are a natural mathematical context for turbulence, and forcing is one of the preferred mechanisms for the generation of the required small scales. The other is boundary friction, which happens to make the vanishing viscosity problem nearly mathematically untreatable. The second motivation is the work by Constantin {\em et al.}, \cite{CDE20}, followed by \cite{NSW21} and \cite{CCS20}, where the authors establish strong convergence of vorticity in the vanishing viscosity limit. In both \cite{CDE20} and \cite{NSW21} the flows are forced and the forcing is allowed to be irregular. In particular, the inviscid vorticity is shown to be a {\em renormalized solution} of the vorticity form of the 2D Euler equations, which is a nonlinear transport equation, and, consequently, enstrophy, along with all $L^p$-norms, are balanced. See also \cite{CS16} and \cite{CNSS17} for earlier, related, work on conservation of $L^p$-norms of vorticity for unforced, vanishing viscosity, inviscid flows  and  \cite{CCS19} for flows generated by the vortex blob method. Finally, our third motivation is that adding forcing turned out to be a mathematically interesting technical problem involving elementary tools.

\section{Preliminaries}

We consider the initial-value problem for the two-dimensional incompressible Euler equations on the torus $\mathbb{T}^2 \equiv [0,2\pi]^2$, with initial velocity $u_0$, subject to a given external force $F$, within a finite time horizon $T>0$:


\begin{equation}\label{eeqF}
\left\{
\begin{array}{ll}
\partial_t u + (u \cdot \nabla) u = -\nabla p + F, & \text{ in } (0,T) \times \torus^2; \\
\dv u = 0, & \text{ in } [0,T) \times \torus^2;\\
u(t=0) = u_0, & \text{ at } \{0\} \times \torus^2.
\end{array}
\right.
\end{equation}

In this work we are interested in {\it weak solutions} for which the vorticity, $\omega \equiv \curl u$, is $p$-th power integrable, for some $p > 1$. Let us begin by recalling the definition of a weak solution.


\begin{definition} \label{wksolee}
Fix $T>0$ and let $u_0 \in L^2(\torus^2)$ be a divergence-free vector field. Assume $F \in L^1((0,T);L^2(\torus^2))$. Let $u \in C^0([0,T);L^2(\mathbb{T}^2))$. We say $u$ is a {\em weak solution of the incompressible Euler equations with forcing term $F$ and  initial velocity $u_0$} if
\begin{enumerate}
\item for every test vector field $\Phi \in C^{\infty}([0,T) \times \torus^2)$ such that $\dv \Phi(t, \cdot) = 0$ the following identity holds true:
\[\int_0^T \int_{\torus^2} \partial_t \Phi \cdot u + u \cdot D\Phi \,u \, {\rm d}x{\rm d}t + \int_{\torus^2} \Phi(0,\cdot) \cdot u_0 \, {\rm d}x + \int_0^T\int_{\torus^2} F \cdot \Phi \,{\rm d}x{\rm d}t = 0.\]

\item For every $t \in [0,T)$, $\mbox{ div } u(t,\cdot) = 0$, in the sense of distributions.
\end{enumerate}
\end{definition}

Existence of weak solutions in the sense of Definition \ref{wksolee} has been established in \cite{D-M1987}, without forcing, for initial velocity fields whose $\curl$, the initial vorticity, is $p^{th}$-power integrable, $p>1$; a simple adaptation of their proof may be used to deal with the case of forcing as well. This is the existence result which is relevant in the context of this work. Other existence results may be found, for instance, in \cite{Delort}, \cite{VW93} and in \cite{DeLellis-Szekelyhidi2019} and references therein.

Our focus hereafter are weak solutions which satisfy an energy balance identity.

\begin{definition} \label{enbalwksol}
We call $u \in C^0([0,T);L^2(\mathbb{T}^2))$ an {\em energy balanced weak solution} if $u$ is a weak solution of the incompressible Euler equations with forcing term $F \in L^1((0,T);L^2(\torus^2))$ for which
\begin{equation}\label{enbalance}
\| u(t,\cdot) \|_{L^2(\torus^2)}^2 = \| u_0 \|_{L^2(\torus^2)}^2 + 2 \int_0^t \int_{\torus^2} F \cdot u \,{\rm d}x{\rm d}s, \text{ for every } t \in [0,T).
\end{equation}
\end{definition}

It is well-known that, for general initial data $u_0 \in L^2$, weak solutions are not unique, see  \cite{DeLellis-Szekelyhidi2019} and references therein. Moreover, uniqueness has not been established even under additional restrictions on integrability of vorticity, with the exception of $\omega_0=\curl u_0 \in L^\infty$, see \cite{Yudovich1963}, and slightly weaker spaces close to $L^\infty$, see \cite{Yudovich1995}, and there is some indication of non-uniqueness of weak solutions if, for any $1\leq p < \infty$, $\omega_0 \in L^p$, see  \cite{Bressan2020}. It is natural, therefore, to focus on special weak solutions, such as those obtained as limits of solutions of the more physically realistic Navier-Stokes equations. It was established in \cite{CLNS16} that, in the absence of forcing,  if the initial vorticity, $\omega_0 \equiv \curl u_0$, belongs to $L^p(\torus^2)$, for
some $p > 1$, then weak solutions which are limits of vanishing viscosity enjoy a special property, namely conservation of kinetic energy.

The initial-value problem for the incompressible 2D Navier-Stokes equations with viscosity $\nu$ and external forcing $F^\nu$, on the torus $\torus^2$, over a finite time horizon $T>0$, is given by:
\begin{equation}\label{nuNSeqF}
\left\{
\begin{array}{ll}
\partial_t u + (u \cdot \nabla) u = -\nabla p + \nu \Delta u + F^\nu, & \text{ in } (0,T) \times \torus^2; \\
\dv u = 0, & \text{ in } [0,T) \times \torus^2;\\
u(t=0) = u_0^\nu, & \text{ at } \{0\} \times \torus^2.
\end{array}
\right.
\end{equation}

It is well-known that, if $u_0 \in L^2(\torus^2)$ and $F^\nu \in L^1((0,T);L^2(\torus^2))$, then \eqref{nuNSeqF} is well-posed in $L^\infty((0,T);L^2(\torus^2)) \cap L^2((0,T);H^1(\torus^2))$; we denote the solution by $u^\nu$. Furthermore, the following energy balance identity holds true:
\begin{eqnarray}\label{enbalanceNS}
\nonumber
\| u^\nu(t,\cdot) \|_{L^2(\torus^2)}^2 = \| u_0^\nu \|_{L^2(\torus^2)}^2 -2\nu \int_0^t \|\omega^\nu({\rm s},\cdot)\|_{L^2(\torus^2)}^2 \,{\rm ds} \\
+ 2 \int_0^t \int_{\torus^2} F^\nu \cdot u^\nu \,{\rm dxds}, \text{ for every } t \in [0,T),
\end{eqnarray}
where $\omega^\nu = \curl u^\nu$.

The term
\begin{equation} \label{dissipterm}
\nu \int_0^t \|\omega^\nu({\rm s},\cdot)\|_{L^2(\torus^2)}^2 \,{\rm ds}
\end{equation}
is called the {\em energy dissipation term} and, formally, it vanishes when $\nu = 0$.


\begin{definition} \label{physsol}
Let $u \in C([0,T);L^2(\torus^2))$ and $F \in L^1((0,T);L^2(\torus^2))$. We say that $u$ is a {\em physically realizable weak solution of the incompressible 2D Euler equations with external forcing} $F$, if the following conditions hold:
\begin{enumerate}
\item[(1)] $u$ is a weak solution of the Euler equations in the sense of Definition \ref{wksolee};
\item[(2)] there exists a family of solutions of the incompressible 2D Navier-Stokes equations with viscosity $\nu >0$, $\{u^{\nu}\}$, with forcing $F^\nu \in L^1((0,T);L^2(\torus^2))$, such that, as $\nu \to 0$,
    \begin{enumerate}
    \item $u^{\nu} \rightharpoonup u$ weakly$\ast$ in $L^{\infty}((0,T);L^2(\torus^2))$;
    \item $F^\nu \rightharpoonup F$ weakly in $L^1((0,T);L^2(\torus^2))$;
    \item $u^{\nu}(0,\cdot)\equiv u_0^{\nu} \to u_0 \equiv u(0,\cdot) $ strongly in $L^2(\torus^2)$.
    \end{enumerate}
\end{enumerate}
\end{definition}

If $u$ is a physically realizable weak solution as above then we refer to a family $\{u^\nu\}$ satisfying (2a)--(2c) of Definition \ref{physsol} as a {\em physical realization} of $u$.

We are now in position to state our main result.

\begin{theorem} \label{physwksoltns}
Let $u \in C([0,T);L^2(\torus^2))$ be a physically realizable weak solution of the incompressible 2D Euler equations with external forcing $F \in L^1((0,T);L^2(\torus^2))$. Consider a physical realization of $u$, $\{u^\nu\}$, which are solutions of the 2D Navier-Stokes equations with viscosity $\nu>0$ and forcing $F^\nu \in L^1((0,T);L^2(\torus^2))$. Suppose, additionally, that, for some $p>1$:
\begin{enumerate}
  \item \label{Lpinitvort} $\curl u_0 \equiv \omega_0 \in L^p(\torus^2)$;
  \item \label{Lpinitvortconv} $\curl u_0^\nu \equiv \omega_0^\nu \to \omega_0$ strongly in $L^p(\torus^2)$;
  \item \label{gnubdd} $g^\nu \equiv \curl F^\nu$ is bounded in $L^1((0,T);L^p(\torus^2)) \cap L^\infty(0,T);L^2(\torus^2))$.
\end{enumerate}

Then $u$ is an energy balanced weak solution.
\end{theorem}

\section{Proof of Theorem \ref{physwksoltns}}

We begin by rewriting the energy balance identity \eqref{enbalanceNS} in a more convenient way. Any solution $u^{\nu}$ of the Navier-Stokes equations \eqref{nuNSeqF} satisfies
\begin{eqnarray}\label{ensandwich}
\nonumber
0 \leq \| u_0^\nu \|_{L^2(\torus^2)}^2 - \| u^\nu(t,\cdot) \|_{L^2(\torus^2)}^2 + 2 \int_0^t \int_{\torus^2} F^\nu \cdot u^\nu \,{\rm dxds} \\
=  2\nu \int_0^t \|\omega^\nu({\rm s},\cdot)\|_{L^2(\torus^2)}^2 \,{\rm ds}.
\end{eqnarray}

Consider $u$ a physically realizable weak solution of the Euler equations, as introduced in Definition \ref{physsol}, and let $\{u^\nu\}$ be a physical realization of $u$. In order to prove Theorem \ref{physwksoltns}  we will establish that $u$ satisfies the inviscid energy balance identity \eqref{enbalance}, by passing to the limit $\nu \to 0$ in \eqref{ensandwich}.

From condition (2a) of Definition \ref{physsol} the family $\{u^\nu\}$ converges to $u$  weak$\ast$-$L^\infty((0,T);L^2(\torus^2))$. In addition, we have $F^\nu \rightharpoonup F$ weak-$L^1((0,T);L^2(\torus^2))$ (condition (2b)) and $u^\nu_0 \to u_0$ strongly in $L^2(\torus^2)$ (this is condition (2c)). Using these facts, and in view of \eqref{ensandwich} it is easy to see that, in order to prove Theorem \ref{physwksoltns}, it is enough to show that the weak-$\ast$ convergence of $u^\nu$ is, in fact, strong, and that the energy dissipation term \eqref{dissipterm} vanishes in the limit. Let us begin by addressing strong convergence.

\begin{lemma} \label{strongconv}
Let $u \in C([0,T);L^2(\torus^2))$ be a physically realizable weak solution with external forcing $F \in L^1((0,T);L^2(\torus^2))$ and let $\{u^\nu\}$ be a physical realization. Assume that hypotheses \ref{Lpinitvort}, \ref{Lpinitvortconv} and \ref{gnubdd} from Theorem \ref{physwksoltns} are satisfied. Then $u^\nu \to u$ strongly in $C([0,T);L^2(\torus^2))$.

\end{lemma}

\begin{proof}

Let $p>1$ be such that \ref{Lpinitvort}, \ref{Lpinitvortconv} and \ref{gnubdd} from Theorem \ref{physwksoltns} are satisfied.

Taking the $\curl$ of the Navier-Stokes equations \eqref{nuNSeqF} leads to the vorticity equations for $\omega^\nu \equiv \curl u^\nu$:
\begin{equation}\label{vortnuNSeqF}
\left\{
\begin{array}{ll}
\partial_t \omega^\nu + (u^\nu \cdot \nabla) \omega^\nu = \nu \Delta \omega^\nu + g^\nu, & \text{ in } (0,T) \times \torus^2; \\
\dv u^\nu = 0, & \text{ in } [0,T) \times \torus^2;\\
\curl u^\nu = \omega^\nu, & \text{ in } [0,T) \times \torus^2;\\
\omega^\nu(t=0) = \omega_0^\nu, & \text{ at } \{0\} \times \torus^2,
\end{array}
\right.
\end{equation}
where $g^\nu = \curl F^\nu$. Standard energy estimates yield $\omega^\nu \in L^\infty((0,T);L^p(\torus^2)) \cap L^2((0,T);W^{1,p}(\torus^2))$. Since $W^{1,p}(\torus^2) \subset L^2(\torus^2)$ if $p>1$ it follows that $\omega^\nu$ is a distributional solution of \eqref{vortnuNSeqF}.

An easy energy estimate then gives
\begin{equation} \label{vortLpest}
\| \omega^\nu(t,\cdot) \|_{L^p(\torus^2)} \lesssim \|\omega_0^\nu\|_{L^p(\torus^2)} + \int_0^T \|g^\nu({\rm s},\cdot)\|_{L^p(\torus^2)} \, {\rm ds}.
\end{equation}
Therefore, using \ref{Lpinitvort}, \ref{Lpinitvortconv} and \ref{gnubdd} from Theorem \ref{physwksoltns}, we find that $\{ \omega^\nu \}$ is a bounded subset of $L^\infty((0,T);L^p(\torus^2))$. It follows by elliptic regularity that $\{u^\nu\}$ is a bounded subset of $L^\infty((0,T);W^{1,p}(\torus^2))$.

From assumption \ref{gnubdd} we have that $\{\curl F^\nu \} $ is bounded in $L^\infty((0,T);L^2(\torus^2))$. Thus the divergence-free part of $F^\nu$, which is the Leray projection $\mathbb{P} F^\nu$, is bounded in $L^\infty((0,T);H^1(\torus^2))$, by the Poincar\'e inequality. This is enough to obtain that $\{u^\nu\}$ is equicontinuous from $[0,T]$ into $H^{-L}(\torus^2)$ for some (perhaps) large $L>0$.

It follows from the Aubin-Lions lemma, since $W^{1,p}(\torus^2)$ is compactly imbedded in $L^2(\torus^2)$, that $\{u^\nu\}$ is compact in $C^0([0,T];L^2(\torus^2))$. Since we already have $u^\nu \rightharpoonup u$ weak$\ast$-$L^\infty((0,T);L^2(\torus^2))$ we deduce convergence of the whole family, just as in \cite[Lemma 1]{CLNS16}.

This concludes the proof.

\end{proof}

Following the strategy set forth before Lemma \ref{strongconv}, it remains to examine the behavior of the dissipation term \eqref{dissipterm}. In the following Proposition we will show that the energy dissipation term vanishes along a subsequence. This is the core of the proof of Theorem \ref{physwksoltns} and the main contribution of the present work. It is elementary, but rather intricate.

\begin{proposition} \label{dissiptermto0}
  Let $u \in C([0,T);L^2(\torus^2))$ be a physically realizable weak solution with external forcing $F \in L^1((0,T);L^2(\torus^2))$ and let $\{u^\nu\}$ be a physical realization. Assume that \ref{Lpinitvort}, \ref{Lpinitvortconv} and \ref{gnubdd} from Theorem \ref{physwksoltns} are satisfied. Then, passing to subsequences as needed,
 \begin{equation} \label{limnuto0dissipterm}
 \lim_{\nu \to 0^+} \nu \int_0^t \|\omega^\nu({\rm s},\cdot)\|_{L^2(\torus^2)}^2 \,{\rm ds} = 0 \text{ for each } t \in [0,T).
 \end{equation}
\end{proposition}

\begin{proof}
Let us assume that, in \ref{Lpinitvort} and \ref{Lpinitvortconv},  $1< p <2$, and that $\omega_0 \notin L^2(\torus^2)$ since, otherwise, the result is trivial.

We divide the proof in several steps. The first step is to establish a differential inequality for $\|\omega^\nu(t,\cdot)\|_{L^2(\torus^2)}$.

\begin{step} \label{step1}
If $z^\nu=z^\nu(t)\equiv \|\omega^\nu(t,\cdot)\|_{L^2(\torus^2)}^2$ then there exist $A$, $B>0$, $\alpha > 2$ such that
\[
\frac{dz^\nu}{dt} \leq -A \nu (z^\nu)^\alpha + B (z^\nu)^{1/2} \text{ for all } \nu>0.
\]
\end{step}

\noindent {\it Proof of Step \ref{step1}:}

\vspace{10pt}

The proof is an adaptation of the proof of estimate (14) in \cite{CLNS16}. We begin with the observation that, if we multiply \eqref{vortnuNSeqF} by $\omega^\nu$ and integrate on the torus $\torus^2$ we get
\begin{equation} \label{L2enerestvort}
\frac{d}{dt} \|\omega^\nu(t,\cdot)\|_{L^2(\torus^2)}^2 = -2\nu \|\nabla\omega^\nu(t,\cdot)\|_{L^2(\torus^2)}^2 +
2\int_{\torus^2} (g^\nu \cdot \omega^\nu)(t,{\rm x}) \,{\rm dx}.
\end{equation}
As in \cite{CLNS16} we use the Gagliardo-Nirenberg inequality to find
\[
-2\nu \|\nabla\omega^\nu(t,\cdot)\|_{L^2(\torus^2)}^2 \leq -2\nu \left(\|\omega^\nu(t,\cdot)\|_{L^2(\torus^2)}^2\right)^{2/(2-p)}
\left(\|\omega^\nu(t,\cdot)\|_{L^p(\torus^2)}\right)^{-2p/(2-p)},
\]
which, in turn, is bounded by
\[
-2C\nu \left(\|\omega^\nu(t,\cdot)\|_{L^2(\torus^2)}^2\right)^{2/(2-p)}
\left(\|\omega^\nu_0\|_{L^p(\torus^2)} + \|g^\nu\|_{L^1((0,T);L^p(\torus^2))}\right)^{-2p/(2-p)},
\]
for some $C=C(p)>0$ by \eqref{vortLpest}.

Now, since $\omega_0^\nu \to \omega_0$ strongly in $L^p(\torus^2)$ (hypothesis \ref{Lpinitvortconv} of Theorem \ref{physwksoltns}) and since $\|g^\nu\|_{L^1((0,T);L^p(\torus^2))}$ is bounded (hypothesis \ref{gnubdd} of Theorem \ref{physwksoltns}), it follows that
\begin{eqnarray} \label{Apart}
-2\nu \|\nabla\omega^\nu(t,\cdot)\|_{L^2(\torus^2)}^2 \leq -A\nu \left(\|\omega^\nu(t,\cdot)\|_{L^2(\torus^2)}^2\right)^{\alpha},
\text{ with } \alpha = \frac{2}{2-p} \text{ and } \\
A=2C\sup_{\nu>0}  \left(\|\omega^\nu_0\|_{L^p(\torus^2)} + \|g^\nu\|_{L^1((0,T);L^p(\torus^2))}\right)^{-2p/(2-p)}.
\end{eqnarray}
We note that $\alpha > 2$.

Let
\[
B = 2\sup_{\nu>0} \| g^\nu \|_{L^\infty((0,T);L^2(\torus^2))}.
\]

We conclude the proof by using \eqref{Apart} together with the Cauchy-Schwarz inequality in \eqref{L2enerestvort} and then switching to the notation
\[z^\nu=z^\nu(t)\equiv \|\omega^\nu(t,\cdot)\|_{L^2(\torus^2)}^2.\]

This concludes the proof of Step \ref{step1}.

\vspace{10pt}

 The next step is a Gronwall-type lemma for the differential inequality obtained in Step \ref{step1}. Before we proceed, we note that the introduction of forcing gives rise to the term $B(z^{\nu})^{1/2}$ in the differential inequality. Our main issue in what follows is to show that this added growth term can be controlled.

Recall that we assumed that $\omega_0$ does not belong to $L^2$. We also assumed that $\omega_0^{\nu} \to \omega_0$ in $L^p$, where $1<p<2$, which gives $\omega_0^\nu \rightharpoonup \omega_0$ in $\mathcal{D}^\prime $. Therefore, by weak lower semicontinuity of the $L^2$-norm, we have $\|\omega^\nu_0\|_{L^2}^2 \equiv z^\nu (0) \to +\infty$ as $\nu \to 0^+$. We emphasize that $\|\omega_0^{\nu}\|_{L^2}$ may or may not be finite, and we wish to consider both possibilities. However, even in the case $\|\omega_0^{\nu}\|_{L^2} = +\infty$, it still holds that
$\|\omega^\nu ( \delta,\cdot)\|_{L^2}<+\infty$ for all $\delta > 0$, due to parabolic regularity.

\begin{step} \label{step2}
Let $\delta \geq 0$ be such that $z^\nu(\delta) < \infty$. Denote by $m=m(t)$ the solution of the ordinary differential equation
\begin{eqnarray} \label{mequation}
 m^\prime = -A\nu m^\alpha + B \sqrt{m}, &   t>\delta, \\
\nonumber m(\delta)= z^\nu (\delta).  &
\end{eqnarray}
Then
\[0 \leq z^\nu(t) \leq m(t) \text{ for all } \delta < t < T.\]
\end{step}

\noindent {\it Proof of Step \ref{step2}:}

\vspace{10pt}

We introduce the notation
\begin{equation}\label{varphi}
  \varphi_\nu=\varphi_\nu(r)=-A\nu r^\alpha + B \sqrt{r}.
\end{equation}
Then we have
\[
m^\prime = \varphi_\nu(m),
\]
\[(z^\nu)^\prime \leq \varphi_\nu(z^\nu),\]
\[m(\delta)=z^\nu(\delta).
\]

It is immediate that $\varphi_\nu$ is strictly concave if $r>0$, so that
\[\varphi_\nu(z^\nu)\leq \varphi_\nu(m) + \varphi_\nu^\prime (m) (z^\nu - m).\]

It follows that
\[\left( z^\nu - m \right)^\prime \leq \varphi_\nu^\prime (m) \left( z^\nu - m \right),\]
and hence
\[\left( z^\nu - m \right)(t) \leq \left( z^\nu - m \right)(\delta) \exp \left( \int_\delta^t \varphi_\nu^\prime (m({\rm s})) \, {\rm ds} \right) \equiv 0,
\]
as desired. This concludes the proof of Step \ref{step2}.

\vspace{10pt}

Next, keeping the notation introduced in the proof of Step \ref{step2}, \eqref{varphi}, we notice that there are only two roots of $\varphi_\nu$: $r=0$ and:
\begin{equation} \label{Rast}
r=R^\ast_\nu \equiv \left(\frac{B}{A\nu}\right)^{2/(2\alpha - 1)}.
\end{equation}
Observe that $\varphi_\nu > 0$ in $(0,R^\ast_\nu)$ and $\varphi_\nu < 0$ in $(R^\ast_\nu,+\infty)$.

The behavior of $m(t)$, and ultimately of $z^\nu (t)$, is determined by the relative position of $m(\delta)$ (which is equal to $z^\nu(\delta)$) with respect to $R^\ast_\nu$. We divide the possibilities into three different cases:
\begin{eqnarray}
& \label{lessthanRast} m(\delta) < R^\ast_\nu: \text{ this implies } m(t) < R^\ast_\nu \text{ for all } t>\delta &
\\ \nonumber & \hspace{2cm} \text{ and, in view of Step \ref{step2}, } z^\nu(t) < R^\ast_\nu  \text{ for all } \delta < t < T; &\\
& \label{equalRast} \hspace{.2cm} m(\delta) = R^\ast_\nu: \text{ this gives } m(t) \equiv R^\ast_\nu \text{ for all } t>\delta \text{ and} &
\\ \nonumber & \hspace{2cm} \text{ using again Step \ref{step2}, }  z^\nu(t) \leq R^\ast_\nu, \text{ for all } \delta < t < T; & \\
& \label{greaterthanRast} \hspace{1cm} m(\delta) > R^\ast_\nu: \text{ in this case we have } m(t) > R^\ast_\nu \text{ for all } t>\delta. &
\end{eqnarray}

Let us now consider three different possibilities for the behavior of $z^\nu (0)$, which, we recall, may or may not be finite: either
\begin{equation}\label{limsupless}
  \limsup_{\nu \to 0^+} \frac{z^\nu (0)}{R^\ast_\nu} < 1, \text{ or}
\end{equation}

\begin{equation}\label{limsupequal}
  \limsup_{\nu \to 0^+} \frac{z^\nu (0)}{R^\ast_\nu} = 1, \text{ or}
\end{equation}

\begin{equation}\label{limsupgreater}
  \limsup_{\nu \to 0^+} \frac{z^\nu (0)}{R^\ast_\nu} > 1.
\end{equation}

In case \eqref{limsupless} holds then there exists a sequence $\{\nu_k\}_{k=1}^{+\infty}$, $\nu_k \to 0^+$, such that, for sufficiently large $k$, $z^{\nu_k}(0) < R^\ast_{\nu_k}$, which we have argued in \eqref{lessthanRast}, leads to $z^{\nu_k}(t) < R^\ast_{\nu_k}$, $0 < t < T$. Recall that $z^{\nu_k}(t)=\|\omega^{\nu_k}(t,\cdot)\|_{L^2(\torus^2)}^2$. We then have
\begin{equation} \label{limsuplessresult}
\lim_{k\to +\infty} \nu_k \int_0^t \|\omega^{\nu_k}({\rm s},\cdot)\|_{L^2(\torus^2)}^2 \, {\rm ds} \leq
\lim_{k\to +\infty} \nu_k \int_0^t \left(\frac{B}{A\nu_k}\right)^{2/(2\alpha-1)} \, {\rm ds} =0,
\end{equation}
since $1-[2/(2\alpha -1)]>0$, as $\alpha > 2$.

Next let us consider the case \eqref{limsupequal}. If this holds true then there exists a sequence $\{\nu_k\}_{k=1}^{+\infty}$, $\nu_k \to 0^+$, such that, for sufficiently large $k$, $z^{\nu_k}(0) \leq  2R^\ast_{\nu_k}$. Now, for those $k$ such that $z^{\nu_k}(0) \leq R^\ast_{\nu_k}$ we already know, from \eqref{lessthanRast} and \eqref{equalRast}, that $z^{\nu_k}(t) \leq R^\ast_{\nu_k}$, $0 < t < T$. If, however, $R^\ast_{\nu_k} < z^{\nu_k}(0) \leq 2R^\ast_{\nu_k}$ then $m(t) > R^\ast_{\nu_k}$, for all $t>0$, see \eqref{greaterthanRast}. But then, since $\varphi_{\nu_k}(m) < 0$, $m(t) \leq m(0) = z^{\nu_k}(0) \leq 2R^\ast_{\nu_k}$. By Step \ref{step2} we obtain, once again, that $z^{\nu_k}(t) \leq 2R^\ast_{\nu_k}$ for all $0< t < T$. Putting these estimates together yields, similarly to the previous case,
\begin{equation} \label{limsupequalresult}
\lim_{k\to +\infty} \nu_k \int_0^t \|\omega^{\nu_k}({\rm s},\cdot)\|_{L^2(\torus^2)}^2 \, {\rm ds} \leq
\lim_{k\to +\infty} \nu_k \int_0^t 2 \left(\frac{B}{A\nu_k}\right)^{2/(2\alpha-1)} \, {\rm ds} =0.
\end{equation}

It remains to analyze the third case, \eqref{limsupgreater}. First we need to introduce more notation:

\begin{equation} \label{Phi}
\Phi_\nu=\Phi_\nu(r) \equiv -\int_r^\infty \frac{{\rm d}\rho}{\varphi_\nu(\rho)} , \text{ with } r>R^\ast_{\nu}.
\end{equation}

We will need some basic facts about $\Phi_\nu$.

\begin{step} \label{step3}
Consider $\Phi_\nu=\Phi_\nu (r)$ given in \eqref{Phi} with $r>R^\ast_\nu$. Then it holds that:
\begin{enumerate}
  \item[(a)] $\Phi_\nu$ is strictly decreasing and, thus, invertible with inverse $\Phi_\nu^{-1}$;
  \item[(b)] \[\lim_{r\to +\infty} \Phi_\nu(r) = 0;\]
  \item[(c)] \[\lim_{r\to \left(R^\ast_\nu\right)^+} \Phi_\nu(r) = +\infty.\]
\end{enumerate}
\end{step}

\noindent {\it Proof of Step \ref{step3}:}

\vspace{10pt}

Item (a) follows immediately from
\[\Phi_\nu^\prime (r) = \frac{1}{\varphi_\nu(r)} < 0 \text{ if } r>R^\ast_\nu.\]

Set
\begin{equation}\label{Rastast}
  R^{\ast\ast}_\nu \equiv \left(\frac{2B}{A\nu}\right)^{2/(2\alpha -1)}.
\end{equation}
Notice that:
\begin{equation} \label{reducetonoforce}
\text{if } r>R^{\ast\ast}_{\nu} \text{ then }\varphi_{\nu}(r) \leq -\frac{A\nu r^\alpha}{2}.
\end{equation}

Item (b) can now be obtained though the estimate below, for $r> R^{\ast\ast}_\nu$:
\[\Phi_\nu(r) \leq \int_r^{+\infty} \frac{2}{A\nu {\rm s}^\alpha} \, {\rm ds} \to 0 \text{ as } r\to +\infty.\]

Finally, we address item (c). We note that
\[\varphi_\nu^{\prime\prime}(r) = -\alpha(\alpha -1)A\nu r^{\alpha -2} - \frac{B}{4 r^{3/2}} < 0,\]
and, therefore, $\varphi_\nu^\prime$ is decreasing. Let $R^\ast_\nu < r < R^{\ast\ast}_\nu$. Observe, also, that $\varphi_\nu^\prime < 0$ in $(R^\ast_\nu,+\infty)$, since $\varphi^\prime_\nu$ is decreasing and vanishes only at $[B/(2A\nu)]^{2/(2\alpha - 1)} < R_\nu^\ast$.

Therefore
\begin{eqnarray*}
& \Phi_\nu(r) \geq - \ds{ \int_r^{R^{\ast\ast}_\nu} \frac{{\rm d}\rho}{\varphi_\nu(\rho)} } &\\
& & \\
& \geq - \ds{ \int_r^{R^{\ast\ast}_\nu} \frac{{\rm d}\rho}{\varphi_\nu^\prime(R^{\ast\ast}_\nu)(\rho - R^\ast_\nu)} } &\\
& \\
& = - \ds{ \frac{1}{\varphi_\nu^\prime(R^{\ast\ast}_\nu) } \log \left| \frac{R^{\ast\ast}_\nu- R^\ast_\nu}{r- R^\ast_\nu}  \right| } & \to +\infty \text{ as } r \searrow R^\ast_\nu.
\end{eqnarray*}
The second inequality above comes from the following set of observations:
\[\varphi_\nu(\rho) = \varphi_\nu^\prime(\theta^\ast_\nu)(\rho - R^\ast_\nu), \text{ for some } R_\nu^\ast \leq \theta_\nu^\ast \leq \rho;\]
\[\varphi^\prime_\nu (R^{\ast\ast}_\nu) < \varphi^\prime_\nu(\theta^\ast_\nu) < 0, \text{ since } R_\nu^\ast \leq \theta_\nu^\ast \leq \rho, \text{ since } \rho < R_\nu^{\ast\ast} \text{ and since } \varphi^{\prime\prime}_\nu < 0.\]
\[\text{Therefore } 0 < -\frac{1}{\varphi^\prime_\nu (R^{\ast\ast}_\nu)} < -\frac{1}{\varphi^\prime_\nu (\theta^{\ast}_\nu)},\]
and, hence, 
\[ -\frac{1}{\varphi_\nu(\rho)} = -\frac{1}{\varphi^\prime_\nu (\theta^{\ast}_\nu)(\rho - R_\nu^\ast)} \geq -\frac{1}{\varphi^\prime_\nu (R^{\ast\ast}_\nu)(\rho - R_\nu^\ast)} \text{ whenever } R_\nu^\ast < \rho < R_\nu^{\ast\ast}. \]

This concludes the proof of Step \ref{step3}.

\vspace{10pt}

It follows from the proof of Step \ref{step3} that $\Phi_\nu$ is a difeomorphism from $(R_\nu^\ast , + \infty)$ to $(0,+\infty)$.

Now we return to the discussion of the case in which $\limsup_{\nu \to 0^+} [z^\nu (0)  / R^\ast_\nu ] > 1$, \eqref{limsupgreater}. Here we can find a sequence $\nu_k \to 0^+$ such that, for sufficiently large $k$, $z^{\nu_k}(0) > R^\ast_{\nu_k}$. Notice that  $z^{\nu_k}(0)=+\infty$ is included here.

Recall that $W^{1,p}(\torus^2)) \subset L^2(\torus^2)$ if $p>1$, and, from the proof of Lemma \ref{strongconv}, that $\omega^\nu \in L^\infty((0,T);L^p(\torus^2)) \cap L^2((0,T);W^{1,p}(\torus^2))$. Thus $\omega^\nu \in L^2((0,T);L^2(\torus^2))$ and, hence, $\omega^\nu(t,\cdot) \in L^2(\torus^2)$ {\it a.e.} $t \in (0,T)$. Recall the $L^2$ energy estimate for vorticity \eqref{L2enerestvort}. Then, from hypothesis \ref{gnubdd} of Theorem \ref{physwksoltns}, it follows that, for $t>s$,
\[\|\omega^\nu(t,\cdot)\|_{L^2}^2 \leq \|\omega^\nu(s,\cdot)\|_{L^2}^2 + C(t-s),\]
where $C=\sup_{\nu>0} \|g^\nu\|_{L^\infty((0,T);L^2(\torus^2))}.$
Therefore, if $\omega^\nu(s,\cdot) \in L^2(\torus^2)$ then $\omega^\nu(t,\cdot) \in L^2(\torus^2)$ for all $t \geq s$. It now follows easily from these observations that $\omega^\nu \in L^\infty_{\loc}((0,T);L^2(\torus^2))$. In addition, using information on $\partial_t \omega^\nu$ given by the vorticity equation \eqref{vortnuNSeqF} we find, with an argument similar to what was already used in Lemma \ref{strongconv}, that $\omega^\nu \in C([0,T);L^p_w(\torus^2))$, where $L^p_w$ refers to the weak topology of $L^p$. Hence, by weak lower semicontinuity of norms, we obtain $z^\nu (0) \leq \liminf_{t \to 0^+} z^\nu (t)$, irrespective of whether $z^\nu(0)$ is finite or not.

In view of these observations together with the choice of $\nu_k$ it follows that  there exist $\delta_k>0$ such that, for all $0<\delta<\delta_k$,  $z^{\nu_k}(\delta)< +\infty$ and  $z^{\nu_k}(\delta) > R^\ast_{\nu_k}$.

Let $m=m(t)$ be the solution of \eqref{mequation} with $0< \delta < \delta_k$ and $m(\delta) = z^{\nu_k}(\delta)$.

We begin by observing that $\Phi_\nu$ and its inverse provide a useful representation formula for $m=m(t)$, the solution of \eqref{mequation}. Indeed, from $m^\prime = \varphi_{\nu_k}(m)$ we find
\[\Phi_{\nu_k} (m(t)) - \Phi_{\nu_k} (m(\delta)) = t - \delta, \text{ for } t>\delta.\]
We note that $\Phi_{\nu_k}$ can be evaluated on $m(t)$ since  $m(t) > R^\ast_{\nu_k}$, $t \geq \delta $, see \eqref{greaterthanRast}.

Therefore,
\[m(t) = \Phi_{\nu_k}^{-1}[t-\delta + \Phi_{\nu_k} (z^{\nu_k}(\delta))],\]
since $m(\delta) = z^{\nu_k}(\delta)$. We note in passing that, for $t>\delta$,   $t-\delta + \Phi_{\nu_k} (z^{\nu_k}(\delta)) \in (0,+\infty)$ and, hence, it belongs to the domain of $\Phi_{\nu_k}^{-1}$ by Step \ref{step3}.

By Step \ref{step2} we deduce that
\[z^{\nu_k}(t) \leq \Phi_{\nu_k}^{-1}[t- \delta + \Phi_{\nu_k} (z^{\nu_k}(\delta))], \quad \text{ for } \delta < t < T.\]

This estimate holds true for all $0<\delta < \delta_k$. We will take  $\liminf_{\delta \to 0^+}$ on both sides of the inequality, recalling that $\liminf_{\delta \to 0^+} = \lim_{\varepsilon \to 0^+} \inf_{0<\delta< \varepsilon}$. We will use the following facts: $\Phi_{\nu_k}$ and $\Phi_{\nu_k}^{-1}$ are decreasing difeomorphisms, and, since $z^{\nu_k}(\delta)>R_{\nu_k}^\ast$ for $0<\delta<\delta_k$ and $(z^{\nu_k})^\prime \leq \varphi_{\nu_k}(z^{\nu_k})$, it follows that, at least in the interval $(0,\delta_k)$, $t \mapsto z^{\nu_k}(t)$ is non-increasing.  Therefore
\begin{align*}
\liminf_{\delta \to 0^+} \Phi_{\nu_k}^{-1}[t- \delta + \Phi_{\nu_k} (z^{\nu_k}(\delta))] & \leq \lim_{\varepsilon \to 0^+} \Phi_{\nu_k}^{-1}[t- \varepsilon + \Phi_{\nu_k} (z^{\nu_k}(0))] \\
&=\Phi_{\nu_k}^{-1}[t + \Phi_{\nu_k} (z^{\nu_k}(0))].
\end{align*}

With this we conclude that
\begin{equation} \label{mainznuest}
z^{\nu_k}(t) \leq \Phi_{\nu_k}^{-1}[t + \Phi_{\nu_k} (z^{\nu_k}(0))], \quad \text{ for } 0 < t < T,
\end{equation}
where, if $z^{\nu_k}(0)=+\infty$ then $\Phi_{\nu_k} (z^{\nu_k}(0))=0$.

\begin{step} \label{step4}
Set $M_k=M_k(t) \equiv  \Phi_{\nu_k}^{-1}[t + \Phi_{\nu_k} (z^{\nu_k}(0))]$. Let $R^{\ast\ast}_{\nu_k}$ be as in \eqref{Rastast}. Then
\begin{equation}\label{blu}
\nu_k \int_0^t z^{\nu_k} ({\rm s}) \, {\rm ds} \leq
\nu_k \int_{R^{\ast\ast}_{\nu_k}}^{z^{\nu_k}(0)} \Phi_{\nu_k}({\rm y}) \, {\rm dy}
+ \nu_k t R^{\ast\ast}_{\nu_k} + \nu_k R^{\ast\ast}_{\nu_k} \Phi_{\nu_k} (z^{\nu_k}(0)),
\end{equation}
for each  $0 < t < T$.
\end{step}

\begin{remark}
This result remains valid even if $z^{\nu_k}(0) \leq R^{\ast\ast}_{\nu_k}$. However, since $z^{\nu_k}(0) \to +\infty$ as $k\to +\infty$, $z^{\nu_k}(0) > R^{\ast\ast}_{\nu_k}$ for sufficiently large $k$.
\end{remark}

\noindent {\it Proof of Step \ref{step4}:}

\vspace{10pt}

From \eqref{mainznuest} we find
\begin{eqnarray}
\label{bla} &\noindent \nu_k \int_0^t z^{\nu_k} ({\rm s}) \, {\rm ds} \hspace{2.5cm} &\\
\label{ble} & \leq
\nu_k \int_0^t  \Phi_{\nu_k}^{-1}[{\rm s} + \Phi_{\nu_k} (z^{\nu_k}(0))] \, {\rm ds}, &
\end{eqnarray}
so that, under the change variables
\[
y = \Phi_{\nu_k}^{-1}[{\rm s} + \Phi_{\nu_k} (z^{\nu_k}(0))]
\]
we deduce
\[ \eqref{ble} = \nu_k \int_{z^{\nu_k}(0)}^{M_k(t)} \frac{{\rm y}}{\varphi_{\nu_k} ({\rm y})} \, {\rm dy}.
\]
Integrating by parts leads to
\begin{eqnarray}
\nonumber \eqref{bla} & \leq \nu_k \int_{M_k(t)}^{z^{\nu_k}(0)} \Phi_{\nu_k}({\rm y})\, {\rm dy} +
\nu_k M_k(t) [t + \Phi_{\nu_k} (z^{\nu_k}(0))] - \nu_k z^{\nu_k}(0)\Phi_{\nu_k}(z^{\nu_k}(0))
\\
\label{bli}  & \leq \nu_k \int^{z^{\nu_k}(0)}_{M_k(t)} \Phi_{\nu_k}({\rm y})\, {\rm dy} +
\nu_k t M_k(t), \hspace{5.5cm}
\end{eqnarray}
since $\Phi_{\nu_k}$ is decreasing, so $M_k(t) \leq z^{\nu_k}(0)$.

Now, if $M_k(t) \leq R^{\ast\ast}_{\nu_k}$ then
\begin{eqnarray*}
  \eqref{bli} & = & \nu_k \int^{R^{\ast\ast}_{\nu_k}}_{M_k(t)} \Phi_{\nu_k}({\rm y})\, {\rm dy} + \nu_k \int^{z^{\nu_k}(0)}_{R^{\ast\ast}_{\nu_k}}  \Phi_{\nu_k}({\rm y})\, {\rm dy} +
\nu_k t M_k(t) \\
   & \leq & \nu_k\Phi_{\nu_k}(M_k(t))(R^{\ast\ast}_{\nu_k} - M_k(t))  + \nu_k \int^{z^{\nu_k}(0)}_{R^{\ast\ast}_{\nu_k}}  \Phi_{\nu_k}({\rm y})\, {\rm dy} +
\nu_k t M_k(t) \\
   &=& \nu_k [t + \Phi_{\nu_k} (z^{\nu_k}(0))] (R^{\ast\ast}_{\nu_k} - M_k(t))  + \nu_k\int^{z^{\nu_k}(0)}_{R^{\ast\ast}_{\nu_k}}  \Phi_{\nu_k}({\rm y})\, {\rm dy} +
\nu_k t M_k(t) \\
   &\leq & \nu_k t R^{\ast\ast}_{\nu_k} + \nu_k \Phi_{\nu_k} (z^{\nu_k}(0)) R^{\ast\ast}_{\nu_k} +
   \nu_k \int^{z^{\nu_k}(0)}_{R^{\ast\ast}_{\nu_k}}  \Phi_{\nu_k}({\rm y})\, {\rm dy},
\end{eqnarray*}
as desired. Here we used that $\Phi_{\nu_k}$ is decreasing and positive.

If, on the other hand, $M_k(t) > R^{\ast\ast}_{\nu_k}$ then
\begin{eqnarray*}
  \eqref{bli} & = & -\nu_k \int_{R^{\ast\ast}_{\nu_k}}^{M_k(t)} \Phi_{\nu_k}({\rm y})\, {\rm dy} + \nu_k \int^{z^{\nu_k}(0)}_{R^{\ast\ast}_{\nu_k}}  \Phi_{\nu_k}({\rm y})\, {\rm dy} +
\nu_k t M_k(t) \\
   & \leq & -\nu_k\Phi_{\nu_k}(M_k(t))(M_k(t) - R^{\ast\ast}_{\nu_k})  + \nu_k \int^{z^{\nu_k}(0)}_{R^{\ast\ast}_{\nu_k}}  \Phi_{\nu_k}({\rm y})\, {\rm dy} +
\nu_k t M_k(t) \\
   &=& -\nu_k [t + \Phi_{\nu_k} (z^{\nu_k}(0))] (M_k(t) - R^{\ast\ast}_{\nu_k})  + \nu_k\int^{z^{\nu_k}(0)}_{R^{\ast\ast}_{\nu_k}}  \Phi_{\nu_k}({\rm y})\, {\rm dy} +
\nu_k t M_k(t) \\
   &\leq & \nu_k t R^{\ast\ast}_{\nu_k} + \nu_k \Phi_{\nu_k} (z^{\nu_k}(0)) R^{\ast\ast}_{\nu_k} +
   \nu_k \int^{z^{\nu_k}(0)}_{R^{\ast\ast}_{\nu_k}}  \Phi_{\nu_k}({\rm y})\, {\rm dy},
\end{eqnarray*}
as we wished.

This completes the proof of Step \ref{step4}.

\vspace{10pt}

Using the estimate of $\varphi_\nu$, \eqref{reducetonoforce}, for $r> R^{\ast\ast}_\nu$, we deduce that
\begin{equation}
\Phi_{\nu_k}(r) \leq \frac{2}{A\nu_k (\alpha -1)r^{\alpha - 1}} \text{ for all } r > R^{\ast\ast}_{\nu_k}.
\end{equation}

Therefore,
\begin{equation} \label{blo}
\nu_k \int^{z^{\nu_k}(0)}_{R^{\ast\ast}_{\nu_k}}  \Phi_{\nu_k}({\rm y})\, {\rm dy} \leq \frac{2}{A(\alpha -1)(\alpha -2)} \left( R^{\ast\ast}_{\nu_k} \right)^{2-\alpha} \to 0 \text{ as } k \to +\infty,
\end{equation}
since $R^{\ast\ast}_{\nu_k} \to +\infty$ and $2-\alpha < 0$.

In addition,
\begin{eqnarray}
\label{grr}  \nu_k  R^{\ast\ast}_{\nu_k}  \to 0 \text{ and }\\
\label{grrr} \Phi_{\nu_k} (z^{\nu_k}(0)) \to 0,
\end{eqnarray}
as $k\to +\infty$, in view of Step \ref{step3} (b) since $z^{\nu_k}(0)\to +\infty$.

Using \eqref{blo}, \eqref{grr} and \eqref{grrr} in \eqref{blu} yields the desired conclusion in the case \eqref{limsupgreater}.

This, together with \eqref{limsuplessresult} and \eqref{limsupequalresult}, concludes the proof of Proposition \ref{dissiptermto0}.

\end{proof}

As previously noted, in view of \eqref{ensandwich}, Theorem \ref{physwksoltns} is an immediate consequence of Lemma \ref{strongconv} and Proposition \ref{dissiptermto0}.

\section{Conclusion}

In this work, the condition that vorticity is $p$-th power integrable plays the role of a compactness criterion. More precisely, a sequence of approximations of velocity with vorticity uniformly in $L^p$, $p>1$, is compact in $L^2$, which implies convergence of the nonlinearity in the Euler equations. In \cite{LMPP20}, Lanthaler {\it et al.} introduced an alternative compactness criterion, expressed in terms of the existence of a uniform modulus of continuity for $L^2$-based structure functions, which the authors prove to be equivalent to both compactness in $L^2$ for the approximate velocities and conservation of energy for the inviscid limit. (See also \cite{CV2018} and \cite{DN2018} for earlier, related, work.) This raises a natural question: what is the extension of their equivalence result to forced flows, and what is the minimum regularity requirement for the forcing. This is a subject of further research.

\vspace{20pt}

{\scriptsize \textbf{Acknowledgments.}  The first author's work was supported by CNPq grant 310441/2018-8 and by FAPERJ grant E-26/202.999/2017. The work of the second author was supported by CNPq grant 309648/2018-1 and by FAPERJ grant E-26/202.897/2018. This work was supported in part by the CNPq-FAPERJ Pronex ``Matem\'atica do movimento dos meios cont\'{\i}nuos e suas
aplica\c{c}\~oes estrat\'egicas".}


\end{document}